\documentclass{amsart}
\usepackage{amsmath,amssymb}
\usepackage{bm}
\usepackage{graphicx}
\usepackage{float}
\usepackage{ascmac}
\usepackage{fancybox}
\usepackage{tabularx}
\usepackage{multicol}

\newtheorem{dfn}{Definition}[section]
\newtheorem{thm}[dfn]{Theorem}
\newtheorem{prop}[dfn]{Proposition}

\newtheorem{remark}[dfn]{Remark}
\newtheorem{example}[dfn]{Example}

\numberwithin{equation}{section}

\title[Local H\"older stabilities for inverse problems]{Local H\"older stabilities for inverse problems of first-order hyperbolic equations}
\author{Giuseppe Floridia}
\address{Università Mediterranea di Reggio Calabria,
Department PAU,
Via dell'Università 25  
89124 Reggio Calabria, Italy}
\email{floridia.giuseppe@icloud.com}

\author{Hiroshi Takase}
\address{Institute of Mathematics for Industry, Kyushu University, 744 Motooka, Nishi-ku, Fukuoka 819-0395, Japan}
\email{htakase@imi.kyushu-u.ac.jp}
\date{July 21, 2022.}
\begin{document}
\begin{abstract}
In this paper, we consider a Cauchy problem for a first-order hyperbolic equation with time-dependent coefficients. Cauchy data are given on a lateral subboundary and we obtain local H\"older stabilities for inverse source and coefficient problems via a Carleman estimate.
\end{abstract}
\maketitle

\section{Introduction and main result}

Let $d\in\mathbb{N}$, $\Omega\subset\mathbb{R}^d$ be a bounded domain with Lipschitz boundary $\partial\Omega$, $T>0$, and $Q:=\Omega\times(0,T)$. We consider the first-order partial differential operator $P$ such that
\[Pu:=A^0(x,t)\partial_tu+A(x,t)\cdot\nabla u,\]
where $A^0\in C^1(\overline{Q})\cap L^\infty(\Omega\times(0,\infty))$ is a positive function, and $A=(A^1,\cdots,A^d)^\mathsf{T}\in C^2(\overline{Q};\mathbb{R}^d)$ is a vector-valued function on $\overline{Q}$. Set
\[\Sigma_+:=\{(x,t)\in\partial\Omega\times(0,T)\mid A(x,t)\cdot\nu(x)>0\},\]
where $\nu$ denotes the outer unit normal to $\partial\Omega$, and $\Sigma_-:=(\Sigma_+)^c=(\partial\Omega\times(0,T))\setminus\Sigma_+$.

For references regarding inverse problems and controllability for first-order hyperbolic equations, see G\"olgeleyen and Yamamoto \cite{Golgeleyen2016}, Cannarsa, Floridia, and Yamamoto \cite{Cannarsa2019}, Floridia and Takase \cite{Floridia2021} and \cite{Floridia}, and the references therein. In particular, Floridia and Takase \cite{Floridia2021} introduced the dissipativeness for vector-valued functions (see Definition 2.4 in \cite{Floridia2021}) and proved global Lipschitz stabilities by observation on $\Sigma_+$ for inverse problems concerning the operator $P$ imposing boundary conditions on $\Sigma_-$. In this paper, without the assumption of such extra boundary conditions on $\Sigma_-$, we obtain local H\"older stabilities for inverse source and coefficient problems. Although a large number of studies have been made on inverse problems for first-order equations, what seems to be lacking is analysis for equations with coefficients depending on both space and time. Indeed, there are few results regarding inverse problems for this kind of equations until \cite{Floridia2021}. In this paper, we investigate inverse source and coefficient problems in a weaker setting than \cite{Floridia2021} in that we do not impose extra boundary conditions on $\Sigma_-$.

Regarding local H\"older stabilities for inverse source and coefficient problems for second-order hyperbolic equations with time-dependent coefficients, readers are referred to Jiang, Liu, and Yamamoto \cite{Liu2017}, Yu, Liu, and Yamamoto \cite{Yu2018}, Bellassoued and Yamamoto \cite{Yamamoto2017}, Klibanov and Li \cite{Klibanov2021}, and Takase \cite{Takase2020}.

\subsection*{Preliminaries}

In the successive two subsections, we present the main results of this paper, that is, inverse source and inverse coefficient problems respectively. Before stating their formulations, we give some notations needed to describe them. In \cite{Floridia2021}, the authors introduced the following definition of dissipativieness for vector-valued functions (see Definition \ref{dissipative}).

\begin{dfn}
For a vector-valued function $X\in C^2(\overline{\Omega};\mathbb{R}^d)$ and $x\in\overline{\Omega}$, a $C^2$ curve $c:[-\eta_1,\eta_2]\to \overline{\Omega}$ for some $\eta_1\ge 0$ and $\eta_2\ge 0$ with $\eta_1+\eta_2>0$ is called an integral curve of $X$ through $x$ if it solves the following initial problem for ordinary differential equations
\[\begin{cases}\displaystyle c'(\sigma):=\frac{dc}{d\sigma}(\sigma)=X(c(\sigma)),\quad \sigma\in[-\eta_1,\eta_2],\\ c(0)=x.\end{cases}\]
\end{dfn}

\begin{dfn}
Let $a,b\in\mathbb{R}$ with $a<b$. An integral curve $c:[a,b]\to\overline{\Omega}$ is called maximal if it cannot be extended to a segment $[a-\varepsilon_1,b+\varepsilon_2]$ for some $\varepsilon_1\ge 0$ and $\varepsilon_2\ge 0$ with $\varepsilon_1+\varepsilon_2>0$ in $\overline{\Omega}$.
\end{dfn}

\begin{dfn}\label{dissipative}
A vector-valued function $X\in C^2(\overline{\Omega};\mathbb{R}^d)$ is called dissipative if, for every $x\in\overline{\Omega}$, the maximal integral curve $c_x$ of $X$ through $x$ is defined on a finite segment $[\sigma_-(x),\sigma_+(x)]$ and $\sigma_-$ can be defined as $\sigma_-\in C(\overline{\Omega})\cap H^2(\Omega)$.
\end{dfn}

Throughout this paper, we assume
\begin{equation}\label{positivity}\exists\rho>0\ \text{s.t.}\ \min_{(x,t)\in\overline{Q}}|A(x,t)|\ge\rho,\end{equation}
\begin{equation}\label{finiteness}A(\cdot,0)\ \text{is dissipative},\end{equation}
and
\begin{equation}\label{spd}\exists C>0\ \text{s.t.}\ \forall\xi\in\mathbb{R}^d,\ \forall(x,t)\in\overline{Q},\quad |\partial_tA(x,t)\cdot\xi|\le C|A(x,t)\cdot\xi|.\end{equation}

\begin{remark}
\eqref{positivity} and \eqref{spd} imply that there exists $\phi\in C^1(\overline{Q})$ such that $A$ can be represented by
\[A(x,t)=A(x,0)e^{\int_0^t \phi(x,s)ds},\quad (x,t)\in\overline{Q}.\]
For the proof, see Proposition 2.10 in \cite{Floridia2021}.
\end{remark}

Owing to \eqref{finiteness}, we can give the following notations. For a fixed $x\in\overline{\Omega}$, let $c_x:[\sigma_-(x),\sigma_+(x)]\to\overline{\Omega}$ be the maximal integral curve of $A(\cdot,0)$, which implies that $c_x$ satisfies
\[\begin{cases}c_x'(\sigma)=A(c_x(\sigma),0),\quad \sigma\in[\sigma_-(x),\sigma_+(x)],\\ c_x(0)=x.\end{cases}\]
Thus, we can define the functions
\begin{equation}\label{distance}\varphi_0(x):=\int_{\sigma_-(x)}^0|c_x'(\sigma)|d\sigma\end{equation}
and
\begin{equation}\label{weight}\varphi(x,t):=\varphi_0(x)-\beta t,\quad (x,t)\in \overline{Q}\end{equation}
for a parameter $\beta>0$ (see also (2.3) and (3.4) in \cite{Floridia2021}). For $\varepsilon\ge 0$, we define
\[\Omega_\varepsilon:=\{x\in\Omega\mid\varphi_0(x)>\varepsilon\}\]
and
\[Q_\varepsilon:=\{(x,t)\in Q\mid \varphi(x,t)>\varepsilon\},\]
where $\varphi\in C(\overline{\Omega})\cap H^2(\Omega)$ (see Lemma 3.2 in \cite{Floridia2021}).

\subsection{Inverse source problem}
For given $g\in H^1(0,T;H^\frac{1}{2}(\partial\Omega))$ and a given open subset $\Sigma\subset \Sigma_+$, we consider the Cauchy problem:
\begin{equation}\label{boundary}\begin{cases}Pu+p(x,t)u=R(x,t)f(x)\quad&\text{in}\ Q,\\
u=g\quad&\text{on}\ \Sigma,\end{cases}\end{equation}
where
\begin{equation}\label{given}p\in W^{1,\infty}(0,T;L^\infty(\Omega)),\ R\in H^1(0,T;L^\infty(\Omega)),\  \text{and}\ f\in L^2(\Omega).\end{equation}
Given $A^0$, $A$, $p$, and $R$, we consider an inverse source problem to determine the source term $f$ in a local domain near $\Sigma$, from lateral boundary data on $\Sigma$ and initial data, and show local H\"older stability. The following theorem is one of the main results of this paper.
\begin{thm}\label{ISP}
Let $g\in H^1(0,T;H^\frac{1}{2}(\partial\Omega))$, $A^0\in C^1(\overline{Q})\cap L^\infty(\Omega\times(0,\infty))$ satisfying $\displaystyle\min_{(x,t)\in\overline{Q}}A^0(x,t)>0$, and $A\in C^2(\overline{Q};\mathbb{R}^d)$ satisfying \eqref{positivity}, \eqref{finiteness}, and \eqref{spd}. Assume \eqref{given} and
\begin{equation}\label{R}\exists m_0>0\ \text{s.t.}\ |R(x,0)|\ge m_0\quad a.e.\ x\in\Omega.\end{equation}
Let $\Sigma\subset\Sigma_+$ be an open subset satisfying
\begin{equation}\label{geometry}\exists\varepsilon_*\ge 0\ \text{s.t.}\ \varnothing\neq Q_{\varepsilon_*}\cap\partial Q\subset \Sigma\cup \big(\Omega\times\{0\}\big).\end{equation}
Then, there exist $\varepsilon^*>0$ such that for any $\varepsilon\in(\varepsilon_*,\varepsilon^*)$, there exist constants $C>0$ and $\theta\in(0,1)$ independent of $f$ such that
\[\|f\|_{L^2(\Omega_\varepsilon)}\le C\Big(\mathcal{D}+\mathcal{F}^{1-\theta}\mathcal{D}^\theta\Big),\]
for all $\displaystyle u\in \bigcap_{k=1}^2H^{k}(0,T;H^{2-k}(\Omega))$ satisfying the Cauchy problem \eqref{boundary}, where
\[\mathcal{F}:=\|f\|_{L^2(\Omega)}+\|u\|_{H^1(0,T;L^2(\Omega))},\quad \mathcal{D}:=\|u(\cdot,0)\|_{H^1(\Omega_{\varepsilon_*})}+\sum_{k=0}^1\|\partial_t^k g\|_{L^2(\Sigma)}.\]
\end{thm}

\begin{example}
Let $d=1$, $\Omega=(0,1)$, and $A(x,t)\equiv 1$. Then, $\Sigma_+=\{1\}\times(0,T)$ and $\varphi(x,t)=x-\beta t$ for $0<\beta<1$. If we set $\Sigma=\Sigma_+$, then \eqref{geometry} holds as it is seen in Figure 1.
\begin{figure}[htbp]
\label{Fig}\centering\includegraphics[scale=0.4]{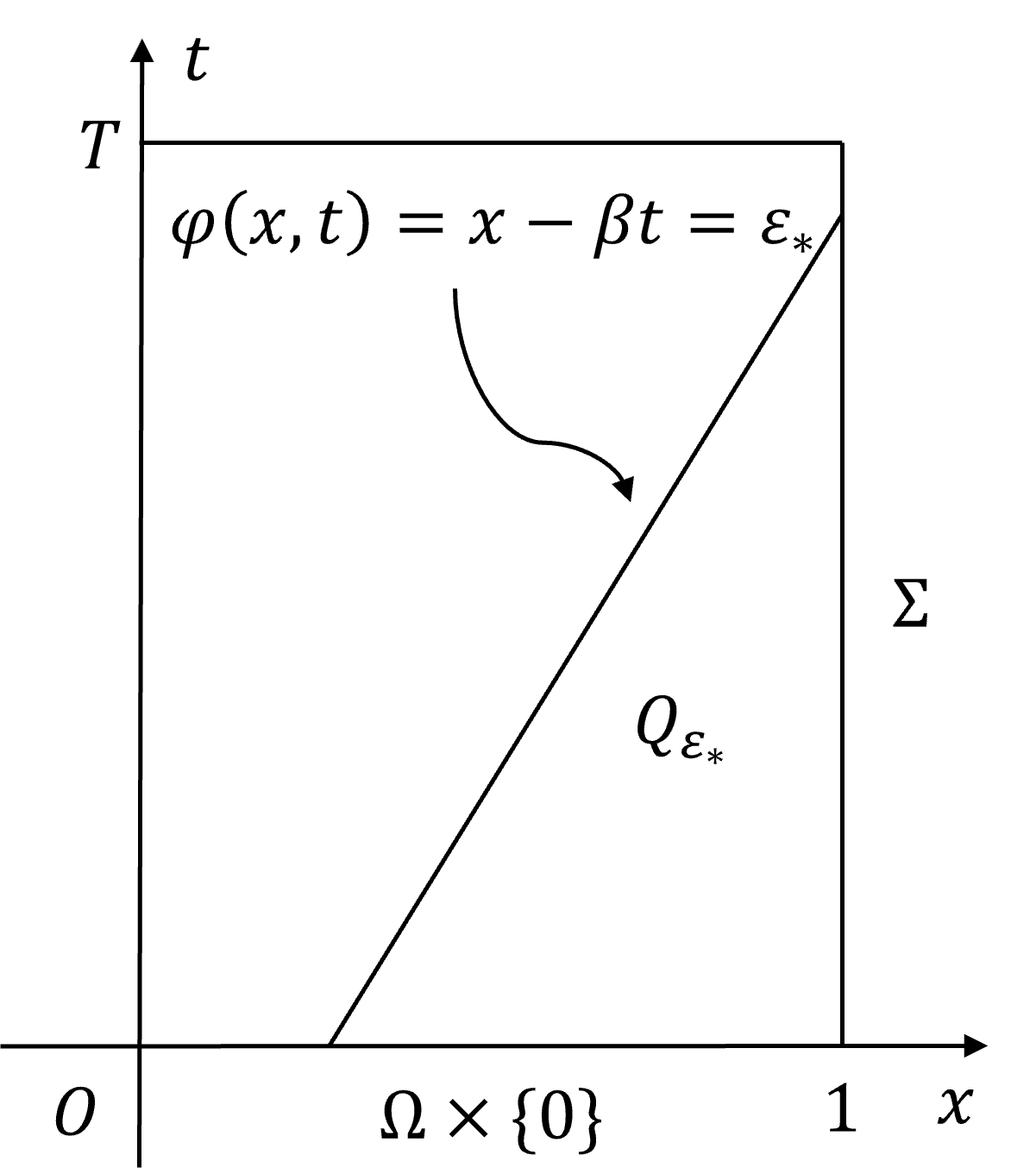}
\caption{On the assumption \eqref{geometry}.}
\end{figure}

\end{example}

\subsection{Inverse coefficient problem}
In this subsection, we assume that $A^0$ and $A$ do not depend on time, i.e., $A^0\in C^1(\overline{\Omega})$ and $A\in C^2(\overline{\Omega};\mathbb{R}^d)$. For given $g\in H^1(0,T;H^\frac{1}{2}(\partial\Omega))$ and a given subset $\Gamma\subset \partial\Omega$, we consider the Cauchy problem:
\begin{equation}\label{boundary2}\begin{cases}Pu+p(x,t)u=0\quad&\text{in}\ Q,\\
u=g\quad&\text{on}\ \Gamma\times(0,T).\end{cases}\end{equation}
Given $p$, we consider an inverse coefficient problem to determine the coefficients $A^0$ and $A$ in a local domain near $\Sigma$ from finitely many lateral boundary data on $\Gamma\times(0,T)$ and initial data, and show local H\"older stability.

For $A\in C^2(\overline{\Omega};\mathbb{R}^d)$, set
\[\Gamma_{+,A}:=\{x\in\partial\Omega\mid A(x)\cdot\nu(x)>0\}\]
and $\Gamma_{-,A}:=\partial\Omega\setminus\Gamma_{+,A}$.

For fixed $M>0$, $\rho>0$, and a subset $\Gamma\subset\partial\Omega$, define the conditional set
\begin{align*}&D(M,\rho,\Gamma)\\
&:=\left\{(A^0,A)\in C^1(\overline{\Omega})\times C^2(\overline{\Omega};\mathbb{R}^d)\middle| \begin{cases}\|A^0\|_{C^1(\overline{\Omega})}+\|A\|_{C^2(\overline{\Omega};\mathbb{R}^d)}\le M,\\
\displaystyle \min_{x\in\overline{\Omega}}A^0(x)\ge\rho,\ \min_{x\in\overline{\Omega}}|A(x)|\ge\rho,\\
\eqref{finiteness},\ \text{and}\ \Gamma\subset\Gamma_{+,A}\ \text{hold}.\end{cases}\right\}.\end{align*}

\begin{thm}\label{ICP}
Let $M>0$, $\rho>0$, $p\in W^{1,\infty}(0,T;L^\infty(\Omega))$, $g_{i,m}\in H^1(0,T;H^\frac{1}{2}(\partial\Omega))$ for $m=1,\ldots,d+1$ and $i=1,2$, $(A^0_i,A_i)\in D(M,\rho,\Gamma)$ for $i=1,2$, and $\Gamma\subset\partial\Omega$ be a subset satisfying
\[\exists\varepsilon_*\ge 0\ \text{s.t.}\ \varnothing\neq Q_{\varepsilon_*}\cap\partial Q\subset (\Gamma\times(0,T))\cup \big(\Omega\times\{0\}\big).\]
Then, there exist $\varepsilon^*>0$ such that for any $\varepsilon\in(\varepsilon_*,\varepsilon^*)$, there exists a constant $C>0$ and $\theta\in(0,1)$ independent of $(A_i^0,A_i)\in D(M,\rho,\Gamma)$ for $i=1,2$ such that
\[\sum_{\mu=0}^d\|A_1^\mu-A_2^\mu\|_{L^2(\Omega_\varepsilon)}\le C\Big(\mathfrak{D}+\mathfrak{F}^{1-\theta}\mathfrak{D}^\theta\Big),\]
for all $\displaystyle u_{i,m}\in \bigcap_{k=1}^2H^k(0,T;W^{2-k,\infty}(\Omega))$ satisfying \eqref{boundary2} with $P=P_i:=A^0_i\partial_t+A_i\cdot\nabla$ and $g=g_{i,m}$,
\[\sum_{k=1}^2\|u_{2,m}\|_{H^k(0,T;W^{2-k,\infty}(\Omega))}\le M\]
for all $m=1,\ldots,d+1$, and
\begin{equation}\label{R2}\exists m_0>0\ \text{s.t.}\ |p(x,0)|\left|\det \begin{pmatrix}u_{2,1}& \cdots&u_{2,d+1}\\ \nabla u_{2,1}& \cdots& \nabla u_{2,d+1}\end{pmatrix}(x,0)\right|\ge m_0\ \text{a.e.}\ x\in\Omega,\end{equation}
where
\[\mathfrak{F}:=\sum_{\mu=0}^d\|A_1^\mu-A_2^\mu\|_{L^2(\Omega)}+\sum_{m=1}^{d+1}\|u_{1,m}-u_{2,m}\|_{H^1(0,T;L^2(\Omega))}\]
and
\[\mathfrak{D}:=\sum_{m=1}^{d+1}\left(\|(u_{1,m}-u_{2,m})(\cdot,0)\|_{H^1(\Omega_{\varepsilon_*})}+\|g_{1,m}-g_{2,m}\|_{H^1(0,T;L^2(\Gamma))}\right).\]

\end{thm}

\subsection{Carleman estimate}

The following Carleman estimate Proposition \ref{Carleman} is decisive to prove the main results. For the proof of Proposition \ref{Carleman}, see Proposition 3.1 and Lemma 3.2 in \cite{Floridia2021}.

\begin{prop}\label{Carleman}
Let $A^0\in C^1(\overline{Q})$ satisfying $\displaystyle\min_{(x,t)\in\overline{Q}}A^0(x,t)>0$, $A\in C^1(\overline{Q};\mathbb{R}^d)$ satisfying \eqref{positivity}--\eqref{spd}, $p\in L^\infty(Q)$, and $\varphi\in C(\overline{\Omega})\cap H^2(Q)$ be the weight function defined by \eqref{weight} satisfying
\[0<\beta<\frac{\rho}{\displaystyle\sup_{x\in\Omega,t>0}A^0(x,t)}.\]
Then, there exist constants $s_*>0$ and $C>0$ such that
\begin{align}\label{estimate}
&s^2\int_{Q} e^{2s\varphi}|u|^2dxdt+s\int_\Omega e^{2s\varphi(x,0)}|u(x,0)|^2dx\\
&\le C\int_{Q} e^{2s\varphi}|(P+p(x,t))u|^2dxdt+Cs\int_{\Sigma_+} e^{2s\varphi}|u|^2dS dt\notag\end{align}
holds for all $s>s_*$ and $\displaystyle u\in \bigcap_{k=0}^1H^k(0,T;H^{1-k}(\Omega))$ satisfying $u(\cdot,T)=0$, where $dS$ denotes the area element of $\partial\Omega$.
\end{prop}


\section{Proof of Theorem \ref{ISP}}
\begin{proof}[Proof of Theorem \ref{ISP}]
For a sufficiently small $\varepsilon>\varepsilon_*$, let $\chi\in C^\infty(\overline{Q})$ be a cutoff function such that
\begin{equation}\label{cutoff}\chi(x,t):=\begin{cases}1,\quad &(x,t)\in \overline{Q_{2\varepsilon}},\\
0,\quad &(x,t)\in \overline{Q\setminus Q_\varepsilon}.\end{cases}\end{equation}
Henceforth, by $C>0$ we denote a generic constant independent of $u$ and $f$ which may change from line to line, unless specified otherwise. Applying the Carleman estimate \eqref{estimate} of Proposition \ref{Carleman} to $\displaystyle\chi\partial_tu\in\bigcap_{k=0}^1H^k(0,T;H^{1-k}(\Omega))$ yields
\begin{align}\label{partial_t3}
&s^2\int_{Q} e^{2s\varphi}|\chi\partial_tu|^2dxdt+s\int_\Omega e^{2s\varphi(x,0)}|\chi(x,0)\partial_t u(x,0)|^2dx\\
&\le C\int_{Q} e^{2s\varphi}|(P+p(x,t))(\chi\partial_tu)|^2dxdt+Cs\int_{\Sigma_+} e^{2s\varphi}|\chi\partial_tu|^2dSdt.\notag\end{align}
Since we obtain
\begin{align*}(P+p(x,t))(\chi\partial_tu)&=\chi\partial_t\Big(A^0(x,t)\partial_tu+A(x,t)\cdot\nabla u+p(x,t)u\Big)\\
&\quad-\chi\partial_tA^0(x,t)\partial_tu-\chi\partial_tA(x,t)\cdot\nabla u-\chi\partial_tp(x,t)u\\
&\quad+[A^0(x,t)\partial_t\chi\partial_tu+A(x,t)\cdot\nabla\chi \partial_tu]\\
&=\chi \partial_tR(x,t)f(x)-\chi\partial_tA^0(x,t)\partial_tu-\chi\partial_tA(x,t)\cdot\nabla u\\
&\quad-\chi\partial_tp(x,t)u+(P\chi)\partial_tu,\end{align*}
we have
\begin{align}\label{above3}|(P+p(x,t))(\chi\partial_tu)|^2&\le C\Big(|\partial_tRf|^2+|\chi\partial_tu|^2+|\chi\partial_tA(x,t)\cdot\nabla u|^2+|\chi u|^2\Big)\\
&\quad +C|P\chi|^2|\partial_tu|^2\notag\\
&\le C\Big(|\partial_tRf|^2+|\chi\partial_tu|^2+|\chi A(x,t)\cdot\nabla u|^2+|\chi u|^2\Big)\notag\\
&\quad +C|P\chi|^2|\partial_tu|^2\notag,\end{align}
where we used \eqref{spd} to obtain the second inequality. Therefore, applying the equation in \eqref{boundary} to the above estimate \eqref{above3} yields
\begin{align}\label{key3}|(P+p(x,t))(\chi\partial_tu)|^2&\le C\Big(|\partial_tRf|^2+|Rf|^2+|\chi\partial_tu|^2+|\chi u|^2\Big)\\
&\quad +C|P\chi|^2|\partial_tu|^2.\notag\end{align}
Applying \eqref{key3} to \eqref{partial_t3} and choosing $s>s_*$ large enough yield
\begin{align}\label{partial_t'3}
&s^2\int_{Q} e^{2s\varphi}|\chi\partial_tu|^2dxdt+s\int_\Omega e^{2s\varphi_0(x)}|\chi(x,0)\partial_t u(x,0)|^2dx\\
&\le C\int_{Q} e^{2s\varphi}\Big(\sum_{k=0}^1|\partial_t^kR|^2\Big)|f|^2dxdt+C\int_{Q}e^{2s\varphi}|\chi u|^2dxdt\notag\\
&\quad+C\int_{Q}e^{2s\varphi}|P\chi|^2|\partial_tu|^2dxdt+Cs\int_{\Sigma} e^{2s\varphi}|\partial_tu|^2dSdt.\notag\end{align}
In regard to the left-hand side of \eqref{partial_t'3}, we obtain
\begin{align}\label{left3}&s^2\int_{Q} e^{2s\varphi}|\chi\partial_tu|^2dxdt+s\int_\Omega e^{2s\varphi_0(x)}|\chi(x,0)\partial_t u(x,0)|^2dx\\
&\ge s\int_\Omega e^{2s\varphi_0(x)}|\chi(x,0)\partial_t u(x,0)|^2dx\notag\\
&\ge Cs\int_\Omega e^{2s\varphi_0(x)}\Big|\chi(x,0)\Big(R(x,0)f(x)-A(x,0)\cdot\nabla u(x,0)\notag\\
&\hspace{200pt}-p(x,0)u(x,0)\Big)\Big|^2dx\notag\\
&\ge Cs\|e^{s\varphi_0}f\|_{L^2(\Omega_{2\varepsilon})}^2-Cs\|e^{s\varphi_0}\nabla u(\cdot,0)\|_{L^2(\Omega_{\varepsilon_*})}^2-Cs\|e^{s\varphi_0}u(\cdot,0)\|_{L^2(\Omega_{\varepsilon_*})}^2\notag\end{align}
for some $C>0$ by \eqref{R}. In regard to the right-hand side of \eqref{partial_t'3}, applying the Carleman estimate \eqref{estimate} of Proposition \ref{Carleman} to $\displaystyle\chi u\in \bigcap_{k=1}^2H^{k}(0,T;H^{2-k}(\Omega))$ yields
\begin{align}\label{right3}
&\int_{Q} e^{2s\varphi}|\chi u|^2dxdt+\frac{1}{s}\int_\Omega e^{2s\varphi_0(x)}|\chi(x,0)u(x,0)|^2dx\\
&\le \frac{C}{s^2}\int_{Q} e^{2s\varphi}|Rf|^2dxdt+\frac{C}{s^2}\int_{Q}|P\chi|^2|u|^2dxdt\notag\\
&\quad+\frac{C}{s}\int_{\Sigma_+} e^{2s\varphi}|\chi u|^2dS dt.\notag\end{align}
Applying \eqref{left3} and \eqref{right3} to \eqref{partial_t'3} yields
\begin{align*}
s\|e^{s\varphi_0}f\|_{L^2(\Omega_{2\varepsilon})}^2&\le Cs\|e^{s\varphi_0}\nabla u(\cdot,0)\|_{L^2(\Omega_{\varepsilon_*})}^2+Cs\|e^{s\varphi_0}u(\cdot,0)\|_{L^2(\Omega_{\varepsilon_*})}^2\\
&\quad+C\int_{Q} e^{2s\varphi}\Big(\sum_{k=0}^1|\partial_t^kR|^2\Big)|f|^2dxdt+\frac{C}{s^2}\int_{Q} e^{2s\varphi}|Rf|^2dxdt\\
&\quad+C\int_{Q}e^{2s\varphi}|P\chi|^2|\partial_tu|^2dxdt+\frac{C}{s^2}\int_{Q}|P\chi|^2|u|^2dxdt\\
&\quad+Cs\int_{\Sigma} e^{2s\varphi}|\partial_tu|^2dSdt+\frac{C}{s}\int_{\Sigma_+} e^{2s\varphi}|\chi u|^2dS dt\end{align*}
and choosing sufficiently large $s>s_*$ yields
\begin{align*}&s\|e^{s\varphi_0}f\|_{L^2(\Omega_{2\varepsilon})}^2\\
&\le C\int_{Q} e^{2s\varphi}\Big(\sum_{k=0}^1|\partial_t^kR|^2\Big)|f|^2dxdt+Cse^{Cs}\|u(\cdot,0)\|_{H^1(\Omega_{\varepsilon_*})}^2\\
&\quad+C\int_{Q}e^{2s\varphi}|P\chi|^2|\partial_tu|^2dxdt+\frac{C}{s^2}\int_{Q}e^{2s\varphi}|P\chi|^2|u|^2dxdt\\
&\quad+Cs\int_{\Sigma} e^{2s\varphi}|\partial_tu|^2dSdt+\frac{C}{s}\int_{\Sigma}e^{2s\varphi}|u|^2dSdt\notag\\
&\le C\int_{Q} e^{2s\varphi}\Big(\sum_{k=0}^1|\partial_t^kR|^2\Big)|f|^2dxdt+Cse^{Cs}\|u(\cdot,0)\|_{H^1(\Omega_{\varepsilon_*})}^2+Ce^{4\varepsilon s}\|u\|_{H^1(0,T;L^2(\Omega))}^2\\
&\quad+Cse^{Cs}\Big(\sum_{k=0}^1\|\partial_t^ku\|_{L^2(\Sigma)}^2\Big)\\
&=C\int_0^T\int_{\Omega_{2\varepsilon}} e^{2s\varphi}\Big(\sum_{k=0}^1|\partial_t^kR|^2\Big)|f|^2dxdt+C\int_0^{T}\int_{\Omega\setminus \Omega_{2\varepsilon}} e^{2s\varphi}\Big(\sum_{k=0}^1|\partial_t^kR|^2\Big)|f|^2dxdt\\
&\quad+Ce^{4\varepsilon s}\|u\|_{H^1(0,T;L^2(\Omega))}^2+Cse^{Cs}\Big(\|u(\cdot,0)\|_{H^1(\Omega_{\varepsilon_*})}^2+\sum_{k=0}^1\|\partial_t^ku\|_{L^2(\Sigma)}^2\Big)\\
&\le C\int_{\Omega_{2\varepsilon}}\left(\int_0^{T}e^{-2s(\varphi_0(x)-\varphi(x,t))}\Big(\sum_{k=0}^1\|\partial_t^kR(\cdot,t)\|_{L^\infty(\Omega)}^2\Big)dt\right)e^{2s\varphi_0}|f|^2dx\\
&\quad+Ce^{4\varepsilon s}\int_{Q}\Big(\sum_{k=0}^1\|\partial_t^kR(\cdot,t)\|_{L^\infty(\Omega)}^2\Big)|f|^2dxdt+Ce^{4\varepsilon s}\|u\|_{H^1(0,T;L^2(\Omega))}^2\\
&\quad+Cse^{Cs}\Big(\|u(\cdot,0)\|_{H^1(\Omega_{\varepsilon_*})}^2+\sum_{k=0}^1\|\partial_t^ku\|_{L^2(\Sigma)}^2\Big)\\
&\le o(1)\|e^{s\varphi_0}f\|_{L^2(\Omega_{2\varepsilon})}^2+Ce^{4\varepsilon s}\Big(\|f\|_{L^2(\Omega)}^2+\|u\|_{H^1(0,T;L^2(\Omega))}^2\Big)\\
&\quad+Cse^{Cs}\Big(\|u(\cdot,0)\|_{H^1(\Omega_{\varepsilon_*})}^2+\sum_{k=0}^1\|\partial_t^ku\|_{L^2(\Sigma)}^2\Big)\end{align*}
as $s\to+\infty$ by the Lebesgue dominated convergence theorem. Choosing $s>s_*$ large enough yields
\begin{align*}\|e^{s\varphi_0}f\|_{L^2(\Omega_{2\varepsilon})}^2&\le Ce^{4\varepsilon s}\Big(\|f\|_{L^2(\Omega)}^2+\|u\|_{H^1(0,T;L^2(\Omega))}^2\Big)\\
&\quad+Ce^{Cs}\Big(\|u(\cdot,0)\|_{H^1(\Omega_{\varepsilon_*})}^2+\sum_{k=0}^1\|\partial_t^ku\|_{L^2(\Sigma)}^2\Big)\\
&\le C\Big(e^{2\varepsilon s}\mathcal{F}+e^{Cs}\mathcal{D}\Big)^2.\end{align*}
Since $\varphi_0(x)>3\varepsilon$ in $\Omega_{3\varepsilon}$, $\|e^{s\varphi_0}f\|_{L^2(\Omega_{2\varepsilon})}^2\ge e^{6\varepsilon s}\|f\|_{L^2(\Omega_{3\varepsilon})}^2$ holds. Then, we obtain
\begin{equation}\label{minimize}\|f\|_{L^2(\Omega_{3\varepsilon})}\le C\Big(e^{-\varepsilon s}\mathcal{F}+e^{Cs}\mathcal{D}\Big)\end{equation}
for sufficiently large $s>s_*$. By replacing $C$ by $Ce^{Cs_*}$, the above estimate holds for all $s>0$. When $\mathcal{D}\ge \mathcal{F}$, \eqref{minimize} implies
\[\|f\|_{L^2(\Omega_{3\varepsilon})}\le Ce^{Cs}\mathcal{D}.\]
On the other hand when $\mathcal{D}<\mathcal{F}$, we choose $s>0$ to minimize the right-hand side of \eqref{minimize} such that
\[e^{Cs}\mathcal{D}=e^{-\varepsilon s}\mathcal{F}\]
i.e.,
\[s=\frac{1}{C+\varepsilon}\log\frac{\mathcal{F}}{\mathcal{D}}.\]
Therefore, we obtain
\[\|f\|_{L^2(\Omega_{3\varepsilon})}\le 2C\mathcal{F}^{1-\theta}\mathcal{D}^\theta,\]
where
\[\theta:=\frac{\varepsilon}{C+\varepsilon}\in(0,1).\]
Hence, there exist constants $C>0$ and $\theta\in(0,1)$ such that
\[\|f\|_{L^2(\Omega_{3\varepsilon})}\leq C\Big(\mathcal{D}+\mathcal{F}^{1-\theta}\mathcal{D}^\theta\Big).\]
\end{proof}

\section{Proof of Theorem \ref{ICP}}
\begin{proof}[Proof of Theorem \ref{ICP}]
Let $\chi$ be the cutoff function defined by \eqref{cutoff}. Henceforth, by $C>0$ we denote a generic constant independent of $u_{i,m}$, $A_i^0$, and $A_i$ which may change from line to line, unless specified otherwise. For $m=1,\ldots,d+1$, setting
\[v_m:=u_{1,m}-u_{2,m},\quad f_1:=A^0_1-A^0_2,\quad f_2:=A_1-A_2,\]
and
\begin{gather*}F:=\begin{pmatrix}f_1\\ f_2\end{pmatrix}\in L^2(\Omega;\mathbb{R}^{d+1}),\\
R_m:=\begin{pmatrix}-\partial_t u_{2,m}& -\partial_{x^1}u_{2,m}&\cdots &-\partial_{x^d}u_{2,m}\end{pmatrix}\in H^1(0,T;L^\infty(\Omega;\mathbb{R}^{d+1})).\end{gather*}
Thus, we obtain
\begin{align*}\begin{cases}P_1v_m+p(x,t)v_m=R_m(x,t)F(x)\quad &\text{in}\ Q,\\
v_m=g_{1,m}-g_{2,m}\quad &\text{on}\ \Gamma\times(0,T),\end{cases}\end{align*}
where the product in the right-hand side of the equation is a product of matrices. Applying the Carleman estimate \eqref{estimate} of Proposition \ref{Carleman} with $P=P_1$ to
\[\chi\partial_tv_m\in\bigcap_{k=0}^1H^{k}(0,T;W^{1-k,\infty}(\Omega))\subset\bigcap_{k=0}^1H^{k}(0,T;H^{1-k}(\Omega))\]
yields
\begin{align*}
&s^2\int_{Q} e^{2s\varphi}|\chi\partial_tv_m|^2dxdt+s\int_\Omega e^{2s\varphi(x,0)}|\chi(x,0)\partial_t v_m(x,0)|^2dx\\
&\le C\int_{Q} e^{2s\varphi}|(P_1+p(x,t))(\chi\partial_tv_m)|^2dxdt+Cs\int_{\Gamma_{+,A_1}\times(0,T)} e^{2s\varphi}|\chi\partial_tv_m|^2dSdt.\end{align*}
Summing up with respect to $m=1,\ldots,d+1$ yields
\begin{align}\label{partial_t2}
&s^2\int_{Q} e^{2s\varphi}|\chi\partial_tv|^2dxdt+s\int_\Omega e^{2s\varphi(x,0)}|\chi(x,0)\partial_t v(x,0)|^2dx\\
&\le C\int_{Q} e^{2s\varphi}|(P_1+p(x,t))(\chi\partial_tv)|^2dxdt+Cs\int_{\Gamma_{+,A_1}\times(0,T)} e^{2s\varphi}|\chi\partial_tv|^2dSdt\notag,\end{align}
where we define
\[v:=\begin{pmatrix}v_1\\ \vdots\\ v_{d+1}\end{pmatrix},\quad R:=\begin{pmatrix}R_1\\ \vdots\\ R_{d+1}\end{pmatrix},\quad (P_1+p(x,t))\partial_tv:=\begin{pmatrix}(P_1+p(x,t))\partial_tv_1\\ \vdots\\ (P_1+p(x,t))\partial_tv_{d+1}\end{pmatrix}.\]
Since we obtain
\begin{align*}(P_1+p(x,t))(\chi\partial_tv_m)&=\chi\partial_t\Big(A^0_1(x)\partial_tv_m+A_1(x)\cdot\nabla v_m+p(x,t)v_m\Big)-\chi\partial_tp(x,t)v_m\\
&\quad +[A^0(x)\partial_t\chi \partial_t v_m+A(x)\cdot \nabla\chi \partial_t v_m]\\
&=\partial_t(R_mF)-\chi\partial_tp(x,t)v_m+(P\chi)\partial_t v_m\end{align*}
for each $m=1,\ldots,d+1$, we have
\begin{equation}\label{key2}|(P_1+p(x,t))(\chi\partial_tv)|^2\le C\Big(|\partial_tRF|^2+|\chi v|^2+|P\chi|^2|\partial_t u|^2\Big).\end{equation}
Applying \eqref{key2} to \eqref{partial_t2} and choosing $s>s_*$ large enough yield
\begin{align}\label{partial_t'2}
&s^2\int_{Q} e^{2s\varphi}|\chi\partial_tv|^2dxdt+s\int_\Omega e^{2s\varphi(x,0)}|\chi(x,0)\partial_t v(x,0)|^2dx\\
&\le C\int_{Q} e^{2s\varphi}|\partial_tRF|^2dxdt+C\int_Qe^{2s\varphi}|\chi v|^2dxdt\notag\\
&\quad+C\int_Q e^{2s\varphi}|P\chi|^2|\partial_tv|^2dSdt+Cs\int_{\Gamma\times(0,T)}e^{2s\varphi}|\partial_t v|^2dSdt.\notag\end{align}
In regard to the left-hand side of \eqref{partial_t'2}, we obtain
\begin{align}\label{left2}&s^2\int_{Q} e^{2s\varphi}|\chi\partial_tv|^2dxdt+s\int_\Omega e^{2s\varphi(x,0)}|\chi(x,0)\partial_t v(x,0)|^2dx\\
&\ge s\int_\Omega e^{2s\varphi(x,0)}|\chi(x,0)\partial_t v(x,0)|^2dx\notag\\
&\ge Cs\int_\Omega e^{2s\varphi(x,0)}\Big|\chi(x,0)\Big(R(x,0)F(x)-A(x)\cdot\nabla v(x,0)\notag\\
&\hspace{200pt}-p(x,0)v(x,0)\Big)\Big|^2dx\notag\\
&\ge Cs\|e^{s\varphi_0}F\|_{L^2(\Omega_{2\varepsilon};\mathbb{R}^d)}^2-Cs\|e^{s\varphi_0}\nabla v(\cdot,0)\|_{L^2(\Omega_{\varepsilon_*};\mathbb{R}^d)}^2\notag\\
&\quad-Cs\|e^{s\varphi_0}v\|_{L^2(\Omega_{\varepsilon_*};\mathbb{R}^d)}^2\notag\end{align}
for some $C>0$ by \eqref{R2}. Indeed, by $\displaystyle\min_{x\in\overline{\Omega}}A^0_2(x)\ge\rho>0$, it follows that
\begin{align*}|\det R(x,0)|&=\left|\det \begin{pmatrix} \partial_tu_{2,1}(x,0)& \cdots& \partial_tu_{2,d+1}(x,0)\\ \nabla u_{2,1}(x,0)& \cdots & \nabla u_{2,d+1}(x,0)\end{pmatrix}\right|\\
&\ge C\left|\det \begin{pmatrix} A_2\cdot\nabla u_{2,1}+p(x,0)u_{2,1}& \cdots& A_2\cdot\nabla u_{2,d+1}+p(x,0)u_{2,d+1}\\ \nabla u_{2,1}(x,0)& \cdots & \nabla u_{2,d+1}(x,0)\end{pmatrix}\right|\\
&= C\left|\det \begin{pmatrix} p(x,0)u_{2,1}(x,0)& \cdots& p(x,0)u_{2,d+1}(x,0)\\ \nabla u_{2,1}(x,0)& \cdots & \nabla u_{2,d+1}(x,0)\end{pmatrix}\right|\\
&= C|p(x,0)|\left|\det \begin{pmatrix} u_{2,1}(x,0)& \cdots& u_{2,d+1}(x,0)\\ \nabla u_{2,1}(x,0)& \cdots & \nabla u_{2,d+1}(x,0)\end{pmatrix}\right|\ge m_0\quad \text{a.e.}\ x\in\Omega.\end{align*}
In regard to the right-hand side of \eqref{partial_t'2}, applying the Carleman estimate \eqref{estimate} of Proposition \ref{Carleman} to $\displaystyle \chi v_m\in\bigcap_{k=1}^2H^{k}(0,T;W^{2-k,\infty}(\Omega))$ for each $m=1,\ldots,d+1$ yield
\begin{align}\label{right2}
&\int_{Q} e^{2s\varphi}|\chi v|^2dxdt\\
&\le \frac{C}{s^2}\int_{Q} e^{2s\varphi}|RF|^2dxdt+\frac{C}{s^2}\int_Q |P\chi|^2|v|^2dxdt\notag\\
&\quad+\frac{C}{s}\int_{\Gamma_{+,A_1}\times(0,T)} e^{2s\varphi}|\chi v|^2dS dt.\notag\end{align}
Applying \eqref{left2} and \eqref{right2} to \eqref{partial_t'2} yields
\begin{align*}
s\|e^{s\varphi_0}F\|_{L^2(\Omega_{2\varepsilon})}^2&\le Cs\|e^{s\varphi_0}\nabla v(\cdot,0)\|_{L^2(\Omega_{\varepsilon_*};\mathbb{R}^d)}^2+Cs\|e^{s\varphi_0}v(\cdot,0)\|_{L^2(\Omega_{\varepsilon_*};\mathbb{R}^d)}^2\\
&\quad+C\int_{Q} e^{2s\varphi}|\partial_tRF|^2dxdt+\frac{C}{s^2}\int_{Q} e^{2s\varphi}|RF|^2dxdt\\
&\quad+C\int_{Q}e^{2s\varphi}|P\chi|^2|\partial_tv|^2dxdt+\frac{C}{s^2}\int_{Q}|P\chi|^2|v|^2dxdt\\
&\quad+Cs\int_{\Gamma\times(0,T)} e^{2s\varphi}|\partial_tv|^2dSdt+\frac{C}{s}\int_{\Gamma\times(0,T)} e^{2s\varphi}|v|^2dS dt\end{align*}
and choosing sufficiently large $s>s_*$ yields
\begin{align*}&s\|e^{s\varphi_0}F\|_{L^2(\Omega_{2\varepsilon};\mathbb{R}^{d+1})}^2\\
&\le C\int_{Q} e^{2s\varphi}\Big(\sum_{k=0}^1|\partial_t^kRF|^2\Big)dxdt+Cse^{Cs}\|v(\cdot,0)\|_{H^1(\Omega_{\varepsilon_*};\mathbb{R}^{d+1})}^2\\
&\quad+Ce^{4\varepsilon s}\|v\|_{H^1(0,T;L^2(\Omega;\mathbb{R}^{d+1}))}^2+Cs e^{C s}\|v\|_{H^1(0,T;L^2(\Gamma;\mathbb{R}^{d+1}))}^2\\
&=C\int_0^T\int_{\Omega_{2\varepsilon}} e^{2s\varphi}\Big(\sum_{k=0}^1|\partial_t^kR|^2\Big)|F|^2dxdt+C\int_0^{T}\int_{\Omega\setminus \Omega_{2\varepsilon}} e^{2s\varphi}\Big(\sum_{k=0}^1|\partial_t^kR|^2\Big)|F|^2dxdt\\
&\quad+Ce^{4\varepsilon s}\|v\|_{H^1(0,T;L^2(\Omega;\mathbb{R}^{d+1}))}^2+Cse^{Cs}\Big(\|v(\cdot,0)\|_{H^1(\Omega_{\varepsilon_*};\mathbb{R}^{d+1})}^2+\|v\|_{H^1(0,T;L^2(\Gamma;\mathbb{R}^{d+1}))}^2\Big)\\
&\le C\int_{\Omega_{2\varepsilon}}\left(\int_0^{T}e^{-2s(\varphi_0(x)-\varphi(x,t))}\Big(\sum_{k=0}^1\|\partial_t^kR(\cdot,t)\|_{L^\infty(\Omega)}^2\Big)dt\right)e^{2s\varphi_0}|F|^2dx\\
&\quad+Ce^{4\varepsilon s}\int_{Q}\Big(\sum_{k=0}^1\|\partial_t^kR(\cdot,t)\|_{L^\infty(\Omega)}^2\Big)|F|^2dxdt+Ce^{4\varepsilon s}\|v\|_{H^1(0,T;L^2(\Omega;\mathbb{R}^{d+1}))}^2\\
&\quad+Cse^{Cs}\Big(\|v(\cdot,0)\|_{H^1(\Omega_{\varepsilon_*};\mathbb{R}^{d+1})}^2+\|v\|_{H^1(0,T;L^2(\Gamma;\mathbb{R}^{d+1}))}^2\Big)\\
&\le o(1)\|e^{s\varphi_0}F\|_{L^2(\Omega_{2\varepsilon};\mathbb{R}^{d+1})}^2+Ce^{4\varepsilon s}\Big(\|F\|_{L^2(\Omega;\mathbb{R}^{d+1})}^2+\|v\|_{H^1(0,T;L^2(\Omega;\mathbb{R}^{d+1}))}^2\Big)\\
&\quad+Cse^{Cs}\Big(\|v(\cdot,0)\|_{H^1(\Omega_{\varepsilon_*};\mathbb{R}^{d+1})}^2+\|v\|_{H^1(0,T;L^2(\Gamma;\mathbb{R}^{d+1}))}^2\Big)\end{align*}
as $s\to+\infty$ by the Lebesgue dominated convergence theorem. Choosing $s>s_*$ large enough yields
\begin{align*}\|e^{s\varphi_0}F\|_{L^2(\Omega_{2\varepsilon};\mathbb{R}^{d+1})}^2&\le Ce^{4\varepsilon s}\Big(\|F\|_{L^2(\Omega;\mathbb{R}^{d+1})}^2+\|v\|_{H^1(0,T;L^2(\Omega;\mathbb{R}^{d+1}))}^2\Big)\\
&\quad+Ce^{Cs}\Big(\|v(\cdot,0)\|_{H^1(\Omega_{\varepsilon_*};\mathbb{R}^{d+1})}^2+\|v\|_{H^1(0,T;L^2(\Gamma;\mathbb{R}^{d+1}))}^2\Big)\\
&\le C\Big(e^{2\varepsilon s}\mathfrak{F}+e^{Cs}\mathfrak{D}\Big)^2.\end{align*}
By the same argument as the proof of Theorem \ref{ISP}, there exist constants $C>0$ and $\theta\in(0,1)$ such that
\[\|F\|_{L^2(\Omega_{3\varepsilon};\mathbb{R}^{d+1})}\leq C\Big(\mathfrak{D}+\mathfrak{F}^{1-\theta}\mathfrak{D}^\theta\Big).\]
\end{proof}

\section*{Acknowledgment}
This work was supported in part by Grant-in-Aid for JSPS Fellows Grant Number JP20J11497, and Istituto Nazionale di Alta Matematica (IN$\delta$AM), through the GNAMPA Research Project 2020, titled ``Problemi inversi e di controllo per equazioni di evoluzione e loro applicazioni'', coordinated by the first author.

\bibliographystyle{plain}
\bibliography{reference}

\end{document}